\documentclass[12pt]{amsart}
\usepackage{amsaddr}
\calclayout
 \usepackage{amsrefs}

\usepackage[T1]{fontenc}       
\usepackage[latin1]{inputenc}  
\usepackage{ae}                
\usepackage{color}             
\usepackage{lscape}            
\usepackage{verbatim}          
\usepackage{epsfig}            
\usepackage{textcomp}          
\usepackage{graphicx,psfrag,amsmath,amsfonts,tabularx,verbatim,
  latexsym,wasysym,amssymb}
\usepackage{float}

\usepackage[toc,page]{appendix}

\newcommand{\e}{\epsilon}
\newcommand{\ue}{u_{\epsilon}}
\newcommand{\uh}{u_h}
\newcommand{\ub}{u_b}

\newcommand{\ur}{u_r}

\newcommand{\aex}{\alpha_{\epsilon}(x)}

\def\div{\operatorname{div}}
\def\O{\Omega}

\def\bxi{{\boldsymbol\xi}}
\def\btau{{\boldsymbol\tau}}

\def\L{\mathcal L}
\def\T{\mathcal T}
\def\N{\mathbb N}
\def\R{\mathbb R}

\def\grad{\operatorname{{\operatorname{\nabla}}}}
\def\curl{\operatorname{\bf{\operatorname{curl}}}}

\date{March 24, 2017}
\begin{document}

\title[A nonlinear RFB Method]
{A Residual-Free Bubble Formulation for nonlinear elliptic problems with oscillatory coefficients}

\author
{Manuel Barreda}
\address{Departamento de Matemática, Universidade Federal do Paran\'a, Curitiba - PR, Brazil} \email{barreda@mat.ufpr.br}

\author
{Alexandre L. Madureira}
\address{
Laborat\'orio Nacional de Computa\c c\~ao Cient\'\i fica, Petr\'opolis - RJ, Brazil} 
\address{
Funda\c c\~ao Get\'ulio Vargas, Rio de Janeiro - RJ, Brazil} 
\email{alm@lncc.br}
\email{alexandre.madureira@fgv.br}
\thanks{Research of the second author was supported by CNPq, Brazil}
\begin{abstract}
We present an investigation of the Residual Free Bubble finite element method for a class of multiscale nonlinear elliptic partial differential equations. After proposing a nonlinear version for the method, we address fundamental questions as existence and uniqueness of solutions. We also obtain a best approximation result, and investigate possible linearizations that generate different versions for the method. As far as we are aware, this is the first time that an analysis for the nonlinear Residual Free Bubble method is considered.
\end{abstract}

\maketitle \numberwithin{equation}{section}

\newtheorem{theorem}{Theorem}[section]
\newtheorem{lemma}[theorem]{Lemma}
\newtheorem{corollary}[theorem]{Corollary}
\newtheorem{proposition}[theorem]{Proposition}
\newtheorem{remark}[theorem]{Remark}

\section{Introduction}\label{s:introdução}
Important physics and engineering problems are nonlinear and of multiscale nature. Examples include certain models for flow in porous media and mechanics of heterogeneous materials. We consider in this work nonlinear elliptic problems of the form
\begin{equation}\label{e:pnl}
-\div[a_\epsilon(x,\ue,\grad\ue)]=f\quad\text{in }\Omega,\qquad\ue=0\quad\text{on }\partial\Omega,
\end{equation}
where $\Omega\subset\mathbb{R}^2$ is a polygonal domain,
\begin{equation*}
a_\epsilon(x,\ue,\grad\ue)=\alpha_\epsilon(x)b(\ue)\grad\ue.
\end{equation*}
and $\alpha_\epsilon$ might have an oscillatory nature. We describe further restrictions on the coefficients latter on.

Problems like~\eqref{e:pnl} are often dealt with using homogenization techniques, even in the linear case. However, this is not always convenient due to restrictive hypothesis on the coefficients, like periodicity or certain probabilistic distributions. Thus, even for the linear situation, several authors developed methods that can compute approximations that do not rely on homogenization. 

It is well-known that standard Galerkin methods perform poorly for such equations, linear or nonlinear, under the presence of oscillatory coefficients~\cites{brezzi,MR2477579}, and there is a strong interest in developing numerical schemes that are efficient for problems with multiscale nature. Important methods include the \emph{Generalized Finite Element Method} (GFEM)~\cite{MR701094},   the \emph{Discontinuous Enrichment Method} (DEM)~\cite{MR1870426}, the \emph{Heterogeneous Multiscale Method} (HMM)~\cite{EW-BE}, and the~\emph{Multiscale Hybrid Mixed Method} (MHM)~\cites{MR3143841,MR3066201,madureira}. We concentrate our literature review on the the \emph{Residual-Free Bubble Method} (RFB)~\cites{MR1222297,B-R,brezzi,MR1159592,MR2203943,MR2142535} and the \emph{Multiscale Finite Element Method} (MsFEM)~\cites{TH-HW,MR2477579,E-H-G,MR1740386,MR1642758,efenpank,MR2448695} since they are closer to our own method. For all the above methodology, the goal is to derive numerical approximations for the multiscale solution using a mesh that is coarser than the characteristic length $\epsilon$ of the oscillations (in opposition to~\cites{SV1,SV2}).

The idea behind the MsFEM is to incorporate local information of the underlying problem into the basis functions of the finite element spaces, capturing microscale aspects. Its analysis was first considered for linear problems, and assuming that the coefficients of the equations are periodic~\cites{MR1642758,MR1740386}. Latter, the non periodic case was also considered~\cite{MR2982460}. An extension for nonlinear problems appears in~\cite{E-H-G}, for pseudo-monotone operators, and the authors show that, under periodicity hypothesis, the numerical solution converges towards the homogenized solution. They also determine the convergence rate if the flux depends only on the gradient of the solution. Further variations of the method were considered in~\cites{chen,CH-Y}. The MHM method shares some of the characteristics of the MsFEM, but so far it was considered only for linear problems.

The HMM approach for linear and nonlinear problems differs considerably, but, as in the MsFEM, the method is efficient in terms of capturing the macroscale behavior of multiscale problems. See~\cites{EW-BE,M-Y} for a description of the method, and~\cite{MR2114818} for a analysis of the method involving linear and nonlinear cases.

The Residual Free Bubble (RFB) formulation~\cites{MR1222297,MR1159592,B-R} was first considered with advection-reaction-diffusion problems in mind. The use of RFB for problems with oscillatory coefficients was already suggested in~\cite{brezzi}, and investigated in~\cite{SG} for the linear case. See~\cite{MR2901822} for a clear description of how the MsFEM and RFB relate. 

In the present work, we extend the RFB formulation for a class of nonlinear problems, with oscillatory coefficients, as in~\eqref{e:pnl}. Such model is a natural extension of the linear problem with oscillatory coefficients, and of the nonlinear problems as considered in~\cite{D-D}, without oscillatory coefficients. We remark that the RFB was considered only in the linear setting, with one exceptions~\cite{ramalho} which considers numerical experiments with RFB for shallow water problem in an ad hoc manner.

Assume that $\alpha_\epsilon(.): \Omega \rightarrow\mathbb{R}$ is measurable, and that there exist positive constants $\alpha_0$ and $\alpha_1$ such that
\begin{equation}\label{limitação.hipótesesH1}
0<\alpha_0\leq\alpha_{\epsilon}(x)\leq\alpha_1\quad\text{almost everywhere in }\Omega.
\end{equation}
Assume also that $b:\mathbb{R}\rightarrow\mathbb{R}$ is continuous and belongs to $W^{2,\infty}(\R)$, and that there exists a constant $b_0$ such that
\begin{equation}\label{limitação.hipótesesH2}
0<b_0\leq b(t) \quad\text{for all }t\in \mathbb{R}.
\end{equation}
Note that a uniform coercivity follows from the above hypothesis, i.e., for almost every $x\in\Omega$, and all $t\in\mathbb{R}$ and $\bxi\in\mathbb{R}^2$,
\begin{equation*}
\alpha_\epsilon (x)b(t)\bxi.\bxi\geq \alpha_0 b_0 \|\bxi\|^2.
\end{equation*}
Rewriting~\eqref{e:pnl} in its variational formulation, we have that $\ue\in H_0^1(\Omega)$ solves
\begin{equation}\label{s:varnonlin.ms}
a(\ue,v)=(f,v)\quad\text{for all }v\in H^1_0(\Omega),
\end{equation}
where
\begin{equation}\label{s:formanonlinear.ms}
a(\psi,\phi)=\int_\Omega\alpha_\epsilon(x)b(\psi)\grad\psi\cdot\grad\phi\,dx.
\end{equation}

Throughout this paper, we denote by $L^2(\Omega)$ the space of square integrable functions, by $W^{q,p}$, $H_0^1(\Omega)$, $H^1(\Omega)$ the usual Sobolev Spaces, and by $H^{-1}(\Omega)$ the dual space of $H_0^1(\Omega)$~\cites{brezis,evans}. By $C$ we denote a generic constant that might have different values at different locations, but that does not depend on $h$ or $\epsilon$.

The outline of the article is as follows. After the introductory Section~\ref{s:introdução}, we describe the RFB method in Section~\ref{s:rfb},
and discuss existence and uniqueness of solutions in Section~\ref{s:eus}. A best approximation result is obtained in Section~\ref{s:melaprox}, and possible linearizations are discussed in Section~\ref{s:lineariz}.

\section{The Residual Free Bubble Method}\label{s:rfb}
Let $\T_h=\{K\}$ be a partition of $\Omega$ into finite elements
$K$, and, associated to $\T_h$, the subspace $V_h\subset
H_0^1(\Omega)$ of piecewise polynomials. The classical finite
element Galerkin method seeks a solution
of~\eqref{s:varnonlin.ms} within $V_h$. The RFB method seeks a
solution within the enlarged, or enriched, space $V_r=V_h\oplus
V_b$, where the bubble space is given
\begin{equation*}
V_b=\{v \in H_0^1(\Omega):\, v|_K \in H_0^1(K) \text{ for all }K \in\T_h\}.
\end{equation*}
That means that we seek $u_r\in V_r$ such that
\begin{equation}\label{e:rfb}
\int_\Omega\alpha_\e(x)b(u_r)\grad(u_r)\cdot\grad v_r\,dx=\int_\Omega fv_r\,dx\quad\text{for all }v_r\in V_r.
\end{equation}
The second equation in the above system is obtained, for each fixed element $K$, by considering $v_r|_K\in H_0^1(K)$ arbitrary and vanishing outside $K$. An integration by parts yield the strong equation of~\eqref{e:rfb}.

This is equivalent to search for $u_r=u_h+u_b$, where $u_h\in V_h$ and $u_b\in V_b$ solve
\begin{equation}\label{e:rfb2}
\begin{gathered}
\int_\Omega\alpha_\e(x)b(u_h+u_b)\grad(u_h+u_b)\cdot\grad v_h\,dx
=\int_\Omega fv_h\,dx\quad\text{for all }v_h\in V_h,
\\
-\div[\alpha_\e(x)b(u_h+u_b)\grad(u_h+u_b)]=f
\quad\text{in }K,\text{ for all }K\in\T_h.
\end{gathered}
\end{equation}
The coupled system~\eqref{e:rfb2} defines the \emph{Residual Free Bubble Method}. The use of bubbles allows the \emph{localization} of the problems of the second equation of~\eqref{e:rfb2}, while the first equation has a global character. Such formulation induces a two-level discretization, where the global problem given by the first equation in~\eqref{e:rfb2} should be discretized by a coarse mesh, and the local problems given by the second equation of~\eqref{e:rfb2} should be solved in a fine mesh. Thus, in terms of computational cost, the first equation is global but posed in a coarse mesh, and the second equation requires refined meshes, but they are local and can be solved in parallel.

Note that for linear problems, it is possible to perform static condensation, ``eliminating'' the bubble part in the final formulation, which is then modified and posed only on the polynomial space~\cites{MR1222297,brezzi,MR1159592,B-R,F-R,SG}. See remark below.

\begin{remark}\label{obslinear}
If $\L$ denotes a linear differential operator, and $a(\cdot,\cdot)$ the associated bilinear form, then it results from the RFB that $u_b\in H_0^1(K)$ solves
\begin{equation*}
\L u_b=-\L u_h+f \quad \text{in }K.
\end{equation*}
Denoting by $\L_K^{-1}: H^{-1}(K)\rightarrow H_0^1(K)$ the local solution operator, we gather that $u_b|_K=\L_K^{-1}(f-\L u_h)$. Thus $u_h \in V_h$ solves that
\begin{equation*}
a(u_h,v_h)+a(\-\sum_{K\in\T_h}\L_K^{-1}\L u_h,v_h)
=(f,v_h)-a(\sum_{K\in\T_h}\L_K^{-1}f,v_h)\quad\text{for all } v_h\in V_h.
\end{equation*}
The formulation above is a perturbed Galerkin formulation. The perturbation aims to capture the microscale effects neglected by coarse meshes.
\end{remark}

\section{Existence and Uniqueness of Solutions}\label{s:eus}
In this section we prove existence and uniqueness results for the continuous problem and for the RFB formulation. We adapt here ideas present in~\cites{boccardo-murat,artola-duvaut}. We shall make use of the following version of the Schauder Fixed-Point Theorem~\cite{diaz-naulin}.
\begin{theorem}[Schauder Fixed-Point Theorem]\label{t:sfp}
Let $E$ be a normed space, $A\subset E$ a non-empty convex set, and $C\subset A$ compact. Then, every continuous mapping $T:A\rightarrow C$ has at least one fixed point.
\end{theorem}
The following result guarantees existence and uniqueness of solutions for the variational problem~\eqref{s:varnonlin.ms}.
\begin{theorem}\label{t:eu}
Let $\alpha_\epsilon(.)$ and $b(.)$ such that~\eqref{limitação.hipótesesH1} and~\eqref{limitação.hipótesesH2} hold. Then, given $f\in L^2(\Omega)$, the variational problem~\eqref{s:varnonlin.ms} has one and only one solution in $H^1_0(\Omega)$.
\end{theorem}
Our proof of Theorem~\ref{t:eu} is based on the lemmata that follow. We first observe
that~\eqref{s:formanonlinear.ms} suggests the
definition
\begin{equation}\label{operador_ponto_fixo_cont}
T^\epsilon:\,L^2(\Omega)\rightarrow H^1_0(\Omega),
\end{equation}
such that, for every $w\in L^2(\Omega)$, the operator $w^\e=T^\e(w)\in H^1_0(\Omega)$ solves
\begin{equation}
\int_\Omega\alpha_\epsilon(x)b(w)\grad w^\epsilon.\grad v\,dx
=\int_\Omega fv\,dx\quad\text{for all } v \in H^1_0(\Omega).
\end{equation}
The operator $T^\epsilon$ is clearly well-defined since, from the hypothesis imposed on $\alpha_{\epsilon}$ and $b$, the bilinear form above satisfies the hypothesis of Lax-Milgram Lemma.
\begin{lemma}\label{lem3.1}
Under the hypothesis of Theorem~\ref{t:eu}, the operator $T^\epsilon$ given by~\eqref{operador_ponto_fixo_cont} is continuous.
\end{lemma}
\begin{proof}
Let $\{w_m\}$ be a sequence in $L^2(\Omega)$ such that $w_m\rightarrow w$ strongly in $L^2(\Omega)$. Consider $T^\epsilon (w_m)=w^\epsilon_m $ and $T^\epsilon(w)=w^\e$. Then,
\begin{gather*}
\int_{\Omega}\alpha_\e(x)b(w_m)\grad w^\epsilon_m\cdot\grad v\,dx
=\int_{\Omega} fv\,dx\quad\text{for all } v \in H^1_0(\Omega),
\\
\int_\Omega\alpha_\e(x)b(w)\grad w^\e\cdot\grad v\,dx
=\int_{\Omega} fv\,dx\quad\text{for all } v \in H^1_0(\Omega).
\end{gather*}
Subtracting both equations, it follows that
\begin{gather*}
\int_{\Omega}\alpha_\e(x)b(w_m)\grad w_m^\e\cdot\grad v\,dx -
\int_{\Omega}\alpha_\e(x)b(w)\grad w^\epsilon\cdot\grad v\,dx=0\quad\text{for all }v \in H^1_0(\Omega).
 \end{gather*}
Adding and subtracting $w^\epsilon$ we gather that
\begin{equation*}
\int_\Omega\alpha_\e(x)b(w_m)\grad(w^\epsilon_m-w^\epsilon+w^\epsilon)\cdot\grad v\,dx =\int_\Omega\alpha_\e(x)b(w)\grad w^\epsilon\cdot\grad v\,dx
\quad\text{for all } v\in H^1_0(\Omega).
\end{equation*}
In an equivalent form, for each $v\in H^1_0(\Omega)$,
\begin{equation}\label{e:estwmw}
\int_{\Omega}\alpha_\e(x)b(w_m)\grad (w^\epsilon_m-w^\epsilon)\cdot\grad v\,dx
=\int_{\Omega}\alpha_{\epsilon}(x)(b(w)-b(w_m))\grad w^{\epsilon}\cdot\grad v\,dx.
\end{equation}
In particular, for $v=w^\epsilon_m-w^\epsilon$ it follows that
\begin{multline*}
\alpha_0 b_0\|\grad(w^\e_m-w^\e)\|_{0,\Omega}^2
\leq\int_{\Omega}\alpha_\e(x)b(w_m)\grad(w^\epsilon_m-w^\epsilon)\cdot\grad(w^\epsilon_m-w^\epsilon)
\\
=\int_{\Omega}\alpha_{\epsilon}(x)(b(w)-b(w_m))\grad w^{\epsilon}\cdot\grad(w^\epsilon_m-w^\epsilon)\,dx
\\
\leq\alpha_1\bigl\|[b(w)-b(w_m)]\grad w^\e\bigr\|_{0,\Omega}\|\grad(w^\e_m-w^\e)\|_{0,\Omega}
\end{multline*}
Thus, $\|\grad(w^\e_m-w^\e)\|_{0,\Omega}\leq C\bigl\|[b(w)-b(w_m)]\grad w^\e\bigr\|_{0,\Omega}$. Now~\cite{artola-duvaut}, since $b(w)-b(w_m)\to0$ in measure, and that $|\grad w^\e|^2\in L^1(\Omega)$, we conclude that $\bigl\|[b(w)-b(w_m)]\grad w^\e\bigr\|_{0,\Omega}\to0$. Thus $w^\e_m\to w^\e$ strongly in $H^1(\Omega)$.
\end{proof}
\begin{lemma}\label{lem3.2} Let $F\in C^1(\mathbb{R})$ such that $F(0)=0$ and $|F'(t)|\leq L$ for all $t\in\mathbb{R}$. Let $\Omega\subset\mathbb{R}^d$ be open, and let $1\leq p<\infty$. Then
\begin{itemize}
\item[a)] if $v\in W^{1,p}(\Omega)$, then $F \circ v \in W^{1,p}(\Omega)$ and  $\partial(F \circ v)/\partial x_i= F'(v)\partial v/\partial x_i$, for $1\leq i\leq d$
\item[b)] if $v\in W^{1,p}_{0}(\Omega)$, then $F \circ v\in W^{1,p}_{0}(\Omega)$.
\end{itemize}
\end{lemma}
\begin{proof} \cite{brezis}*{Proposition 9.5}.
\end{proof}
\begin{lemma}\label{l:contpf}
Under the hypotheses of Theorem~\ref{t:eu}, the uniqueness of solutions for~\eqref{e:pnl} follows.
\end{lemma}
\begin{proof}
Let, for $t\in\mathbb{R}$,
\begin{equation*}
\tilde b(t)=\int_0^t\,b(s)ds.
\end{equation*}
Since $b\in C^0(\mathbb{R})$, then $\tilde b\in C^1(\mathbb{R})$. Moreover, $\tilde b'$ is always positive, and then $\tilde b$ is a bijection in $\mathbb{R}$. Consider the Kirchhoff transform $U_\e=\tilde b(\ue)$. From Lemma~\ref{lem3.2} we gather that
\begin{equation*}
\grad U_\e= b(\ue)\grad\ue
\end{equation*}
and $U_\epsilon\in H^1_0(\Omega)$. Thus,~\eqref{e:pnl} is equivalent to the linear problem
\begin{equation}\label{c:probmodelequivalin.ms}
\begin{gathered}
-\div[\alpha_\epsilon(x)\grad U_\epsilon]=f\quad\text{in }\Omega,
\\
U_\epsilon=0\quad\text{on }\partial\Omega.
\end{gathered}
\end{equation}
that is, $\ue $ solves~\eqref{e:pnl} is and only if $U_\epsilon$ solves~\eqref{c:probmodelequivalin.ms}.

From Lax-Milgram Lemma, there is at most one solution for~\eqref{c:probmodelequivalin.ms}, and therefore, there is also at most one solution for~\eqref{e:pnl}. Indeed, if there were two solutions for~\eqref{e:pnl}, we would be able to construct also two solutions for~\eqref{c:probmodelequivalin.ms}.
\end{proof}
We now prove Theorem~\ref{t:eu}.
\begin{proof}[Existence]
Consider in Theorem~\ref{t:sfp} that $A=E=L^2(\Omega)$, $C= H^1_0(\Omega)$, and the operator $T^{\epsilon}$ defined by~\eqref{operador_ponto_fixo_cont}. Then, from Lemma~\ref{lem3.1} we conclude that $T^{\epsilon}$ has a fixed point.
\end{proof}
\begin{proof}[Uniqueness] Follows from Lemma~\ref{l:contpf}.
\end{proof}

To show existence of the RFB solution, it is enough to pursue the same ideas just presented, but now considering the operator
\begin{equation*}
T_h^\epsilon:\,L^2(\Omega)\rightarrow V_r,
\end{equation*}
where, for a given $w\in L^2(\Omega)$, we define $w_r^\e=T_h^\epsilon(w)$ such that
\begin{equation*}
\int_\Omega\alpha_\epsilon(x)b(w)\grad w_r^\epsilon.\grad v\,dx
=\int_\Omega fv\,dx\quad\text{for all } v \in V_r.
\end{equation*}
As in Lemma~\ref{lem3.1}, the operator $T_h^\epsilon$ is continuous. The proof is basically the same, replacing $H_0^1(\Omega)$ by $V_r$.

\begin{remark} In~\cite{efenpank} the existence and uniqueness result for solutions for the MsFEM requires monotonicity. Such results were obtained~\cite{xu} without monotonicity assumptions, but under the condition that the discrete and exact solutions are close. We follow the same approach.
\end{remark}

To establish a uniqueness result, let $\L u=-\div[\alpha_\e(x)b(u)\grad u]$, and its Fréchet derivative in $u$ defined by
\[
 \L'(u)v=-\div\{\alpha_\e(x)\grad[b(u)v]\}=-\div\{\alpha_\e(x)[b(u)\grad v+b'(u)v\grad u].
\]
Consider also~\eqref{s:formanonlinear.ms} and
\[
 a'(u;v,\chi)=\int_\O\alpha_\e\grad[b(u)v]\cdot\grad\chi=\int_\O\alpha_\e[b(u)\grad v\cdot\grad\chi+b'(u)v\grad u\cdot\grad\chi],
\]
induced by $\L$ and $\L'$ respectively. From~\cite{PR}*{Theorem 6 and Remark 6}, it follows that $\L'(u)$ defines an isomorphism from $H_0^1(\O)$ in $H^{-1}(\O)$. Note that if $\chi=b(u)v$, then
\begin{equation*}
\sup_{\chi\in H_0^1(\O)}\frac{a'(u;v,\chi)}{\|\chi\|_1}\ge\frac{\int_\O\alpha_\e|\grad[b(u)v]|^2}{\|b(u)v\|_1}
\ge\alpha_0\|b(u)v\|_1\ge c(u)\|v\|_1.
\end{equation*}
Note also that
\begin{equation*}
|b(u)v|_1=\|b(u)\grad v + b'(u)v\grad u\|_0
\geq\|b(u)\grad v\|_0 - \| b'(u)v\grad u\|_0,
\end{equation*}
and, on the other hand, from Poincaré's inequality,
\begin{equation*}
\| b'(u)v\grad u\|_0 \leq \|\grad b(u)\|_{L^\infty(\O)}\|v\|_0
\leq C_\O\|\grad b(u)\|_{L^\infty(\O)}\|\grad v\|_0.
\end{equation*}
It is enough to consider then
\begin{equation*}
c(u)\ge\alpha_0(b_0-C_\O\|\grad b(u)\|_{L^\infty(\O)}).
\end{equation*}
Thus, for $\|u\|_{1,\infty}$ sufficiently small, $c(u)$ is positive.

In what follows, we consider the Galerkin projection $P_h:H_0^1(\O)\to V_r$ with respect to the bilinear form $\int_\O\aex b(\ue)\grad v\grad\chi\,dx$. Assume also that
\begin{equation*}
\|\chi-P_h\chi\|_{L^2(\O)}\le \hat c(h)\|\chi\|_{H^1(\O)},
\end{equation*}
where $\hat c(h)\to0$ independently of $\e$. This holds, for instance, if $\alpha(\cdot)$ is $\e$-periodic~\cite{CH-Y}.

Consider the following result.
\begin{lemma}
Let $u$ and $\tilde u\in H^1(\O)$. Then
\begin{equation}\label{e:coerc}
\bar c(u)\|v_h\|_1\le\sup_{\chi_h\in V_h}\frac{a'(\tilde u;v_h,\chi_h)}{\|\chi_h\|_1},
\end{equation}
where $\bar c(u)=c(u)-\hat c(h)-\|u-\tilde u\|_{1,\infty}\|b\|_{2,\infty}\|\alpha\|_{0,\infty}\|u\|_{1,\infty}$.
\end{lemma}
\begin{proof}
To show~\eqref{e:coerc}, note that
\begin{multline*}
a'(\tilde u;v_h,\chi_h)
=\int_\O\alpha_\e[b(\tilde u)\grad v_h\cdot\grad\chi_h+b'(\tilde u)v_h\grad\tilde u\cdot\grad\chi_h]
\\
=\int_\O\alpha_\e\{b(u)\grad v_h\cdot\grad\chi_h
+b'(u)v_h\grad u\cdot\grad\chi_h\}
\\
+\int_\O\alpha_\e\{[b(\tilde u)-b(u)]\grad v_h\cdot\grad\chi_h
+[b'(\tilde u)\grad\tilde u-b'(u)\grad u]v_h\cdot\grad\chi_h\}
\\
\ge a'(u;v_h,\chi_h)-\delta\|v_h\|_1\|\chi_h\|_1
\end{multline*}
where
\[
 \delta=\|\alpha\|_{0,\infty}\|b\|_{2,\infty}\|\tilde u\|_{1,\infty}\|\tilde u-u\|_{1,\infty}.
\]

Observe that, from~\cite{xu}*{Lemma 2.2},
\[
 a'(u;v_h,P_h\chi)\ge a'(u;v_h,\chi)-\hat c(h)\|v_h\|_1\|\chi\|_1
\]
for all $\chi\in H_0^1(\O)$. Then
\begin{multline*}
 \sup_{\chi_h\in V_r}\frac{a'(\tilde u;v_h,\chi_h)}{\|\chi_h\|_1}
=\sup_{\chi\in H_0^1(\O)}\frac{a'(\tilde u;v_h,P_h\chi)}{\|P_h\chi\|_1}
\\
\ge c\sup_{\chi\in H_0^1(\O)}\frac{a'(u;v_h,P_h\chi)}{\|\chi\|_1}-\delta\|v_h\|_1
\ge\sup_{\chi\in H_0^1(\O)}\frac{a'(u;v_h,\chi)}{\|\chi\|_1}-(\hat c(h)+\delta)\|v_h\|_1
\\
\ge[c(u)-\hat c(h)-\delta]\|v_h\|_1\ge\bar c(u)\|v_h\|_1
\end{multline*}
for $\delta$ and $h$ sufficiently small. Above, we use the inequality $\|P_h\chi\|_1\le c\|\chi\|_1$.
\end{proof}
\begin{theorem}
Let $u_h$ and $\tilde u_h$ be two solutions for the  discrete problem such that
\[
 \|u-u_h\|_{1,\infty}+\|u-\tilde u_h\|_{1,\infty}\le\eta,
\]
where $\eta$ is small enough. Then $u_h=\tilde u_h$.
\end{theorem}
\begin{proof}
Note that
\[
\|u-u_h-t(\tilde u_h-u_h)\|\le(1-t)\|u-u_h\|+t\|u-\tilde u_h\|\le\eta,
\]
for all $t\in[0,1]$. Let $\eta$ be small enough such that
\[
 \bar c(u)=c(u)-\hat c(h)-\eta\|b\|_{2,\infty}\|\alpha\|_{0,\infty}\|u\|_{1,\infty}>0.
\]
Then
\begin{multline*}
\bar c(u)\|u_h-\tilde u_h\|_1=\bar c(u)\int_0^1\|u_h-\tilde u_h\|_1\,dt
\\
\le \int_0^1\sup_{\chi_h\in V_h}\frac{a'(u_h+t(\tilde u_h-u_h);u_h-\tilde u_h,\chi_h)}{\|\chi_h\|_1}\,dt
\\
\le\sup_{\chi_h\in V_h}\frac{\int_0^1a'(u_h+t(\tilde u_h-u_h);u_h-\tilde u_h,\chi_h)\,dt}{\|\chi_h\|_1}
\\
=\sup_{\chi_h\in V_h}\frac{\int_0^1\frac d{dt}a(u_h+t(\tilde u_h-u_h),\chi_h)\,dt}{\|\chi_h\|_1}=0.
\end{multline*}
Since $\bar c(u)>0$, then $u_h=\tilde u_h$.
\end{proof}

\section{Best approximation result}\label{s:melaprox}
We establish here a Céa's Lemma type result for the Residual Free Bubble Method. The strategy to obtain such result is to consider a linearization $A(u_r;\cdot,\cdot)$ of~\eqref{s:formanonlinear.ms} centered at the ``enriched solution'' $u_r$. We consider then the following linear problem to find $w\in H_0^1(\O)$ such that  \begin{equation*}
A(u_r ;w,v)=(f,v)\quad\text{for all } v \in H^1_0(\Omega),
 \end{equation*}
where
\begin{equation*}
A(u_r;w,v)=\int_\O\aex\,b(u_r )\grad w\cdot\grad v\,dx.
\end{equation*}
Thus, $A(u_r;\cdot,\cdot)$ is coercive in $H_0^1(\O)$, since
\begin{equation}\label{e:coercA}
A(u_r;w,w)=\int_\O\aex\,b(u_r )|\grad w|^2\,dx
\ge C_\O\alpha_0b_0\|w\|_{1,\O}^2,
\end{equation}
where $C_\O$ is the Poincaré's constant.

We establish first the following identity.
\begin{lemma}
Given $v_r \in V_r$, the following identity holds
\begin{equation}\label{s:identidade.cea}
A(u_r; \ue-u_r,v_r)=A(u_\e; u^\e-u_r,v_r)=
\int_{\Omega}\alpha_{\epsilon}(x)[b(u_r)- b(\ue)]
\grad \ue\cdot\grad v_r\, dx.
\end{equation}
\end{lemma}
\begin{proof} Indeed,
\begin{multline*}
A(u_r; \ue-u_r,v_r)=\int_{\Omega}\alpha_{\epsilon}(x)b(u_r)\grad \ue.\grad v_r\, dx
-\int_{\Omega}\alpha_{\epsilon}(x)b(u_r) \grad u_r.\grad v_r\,dx
\\
= \int_{\Omega}\alpha_{\epsilon}(x)b(u_r)\grad \ue.\grad v_r \, dx
- \int_{\Omega} f v_r\,dx
\\
= \int_{\Omega}\alpha_{\epsilon}(x)b(u_r)\grad \ue.\grad
v_r \, dx -\int_{\Omega}\alpha_{\epsilon}(x)b(\ue)\grad
\ue.\grad v_r \, dx
\\
= \int_\O\alpha_\e(x)[b(u_r)-b(u^\e)]\grad u^\e.\grad v_r \, dx.
\end{multline*}
The proof of the second inequality is similar.
\end{proof}
We end the present section establishing a best approximation result in the enriched space $V_r$. This is a Céa's Lemma type result for the multiscale nonlinear problem~\cite{BS}. An advantage of the estimate is that it requires less regularity of $b(\cdot)$ than in~\cite{D-D}, cf. also Remark~\ref{r:dd}.

We often use H\"older's inequality
\[
 \int_\O fgh\,dx\le\|f\|_{L^3}\|g\|_{L^6}\|h\|_{L^2}
\le\|f\|_{0,\O}^{1/2}\|f\|_{1,\O}^{1/2}\|g\|_{L^6}\|h\|_{L^2}
\]
where we use also the continuous embedding $H^1(\O)\hookrightarrow L^6(\O)$ (for dimensions smaller than three).
\begin{proposition}\label{l:cea}
Let $\alpha_\epsilon(.)$ and $b(.)$ satisfying~\eqref{limitação.hipótesesH1} and~\eqref{limitação.hipótesesH2}, respectively. Then, for $\ue$ sufficiently small in $W^{1,6}(\Omega)$, it follows that \begin{equation}\label{e:cea}
\|\grad(\ue - u_r)\|_{0,\Omega} \leq C  \|\grad(\ue-w_r)\|_{0, \Omega}\quad\text{for all } w_r \in V_r.
\end{equation}
\end{proposition}
\begin{proof}
Let $w_r\in V_r$. To establish~\eqref{e:cea}, compute
\begin{multline}\label{eq.cea}
A(u_r;\ue-u_r,\ue-u_r) =A(u_r;\ue -
u_r,\ue - w_r) + A(u_r;\ue - u_r,w_r - u_r)
\\
=\int_\O\alpha_\e\,b(u_r)\grad(u^\e-u_r)\cdot\grad(u^\e-w_r)\,dx
+ \int_\O\alpha_\e\,(b(u_r)-b(\ue))\grad u^\e\cdot\grad(w_r-u_r)\,dx
\end{multline}
using~\eqref{s:identidade.cea}. Denote by $I_1$, $I_2$ the first and second terms of~\eqref{eq.cea}. We now estimate each of these terms
\begin{equation*}
I_1=\int_\Omega\alpha_\e\,b(u_r)\grad(u^\e-u_r)\cdot\grad(\ue-w_r)\,dx
\leq c_1\|\grad(\ue-u_r)\|_{0,\Omega}
\|\grad(\ue-w_r)\|_{0,\Omega},
\end{equation*}
where $c_1:=\alpha_1\|b\|_\infty$. We estimate now $I_2$:
\begin{multline*}
I_2=\int_\Omega\alpha_\epsilon\,(b(u_r)-b(\ue))\grad
\ue .\grad(w_r-u_r)\,dx
\\
\leq\alpha_1\|b'\|_{\infty}\int_\Omega|u_r-\ue|\,|\grad
\ue| \,|\grad(w_r-u_r)|\,dx
\\
\leq\alpha_1\|b'\|_{\infty}\|\ur-\ue\|_{L^2(\O)}^{1/2}\|\ur-\ue\|_{H^1(\O)}^{1/2}
\|\grad\ue\|_{L^6(\O)}\|w_r-u_r\|_{1,\O}
\\
\leq\alpha_1\|b'\|_{\infty}\|\grad
\ue\|_{L^6(\O)}\|\ue-\ur\|_{H^1(\O)}
\bigl[\|\ue-\ur\|_{1,\O}+\|\ue-w_r\|_{1,\O}\bigr].
\end{multline*}
From~\eqref{e:coercA}, there exists $\beta>0$, independent of $\epsilon$, such that
\begin{gather*}
A(u_r;\ue-u_r,\ue-u_r)\geq\beta\|\grad(u^\e-u_r)\|^2_{0,\O}.
\end{gather*}
Moreover, from the estimates for $I_1$, $I_2$ in~\eqref{eq.cea}, we gather that
\begin{multline*}
\beta\,\|\grad(\ue-u_r)\|^2_{0,\Omega}\leq
c_1\|\grad(\ue-u_r)\|_{0,\Omega}\|\grad(\ue-w_r)\|_{0,\Omega}
\\
+\alpha_1\|b'\|_{\infty}\|\grad
\ue\|_{L^6(\O)}\|\ue-\ur\|_{H^1(\O)}
\bigl[\|\ue-\ur\|_{1,\O}+\|\ue-w_r\|_{1,\O}\bigr].
\end{multline*}
Thus
\begin{multline*}
\beta\,\|\grad(\ue-u_r)\|_{0,\Omega}\leq
\bigl(c_1+\alpha_1\|b'\|_{\infty}\|\grad\ue\|_{L^6(\O)}\bigr)\|\ue-w_r\|_{H^1(\Omega)}
\\
+
\alpha_1\|b'\|_{\infty}\|\grad\ue\|_{L^6(\O)}\|\ue-\ur\|_{1,\O},
\end{multline*}
and then
\begin{equation*}
\|\ue-u_r\|_{1,\Omega}
\bigl(\beta-\alpha_1\|b'\|_{\infty}\|\grad\ue\|_{L^6(\O)}\bigr)
\leq c_1\|\ue-w_r\|_{1,\O}.
\end{equation*}
\end{proof}
\begin{remark}
Proposition~\ref{l:cea} is important because the best approximation estimate is independent of $\e$, and shows in particular that the RFB method converges at least as well as the MsFEM since the RFB approximation spaces contains the spaces employed in the MsFEM. The choice of the approximation spaces is crucial here, since polynomial spaces with no bubbles added, a.k.a. classical Galerkin, yield a method that converges in $h$ albeit non-uniformly with respect to $\e$. 
\end{remark}

\begin{remark}\label{r:dd}
Dropping the ``small solution'' hypothesis,  (also present in~\cite{AV}), an analogous result holds. In particular, the estimate
\begin{equation*}
 \|\ue-\ur\|_{H^1(\O)} \leq C\bigl(\|\ue-w_r\|_{H^1(\O)}+\|\ue-\ur\|_{L^2(\O)}\bigr)
\quad\text{for all }w_r \in V_r
\end{equation*}
results from the above proof. An estimate for
$\|\ue-\ur\|_{H^1(\O)}$ was obtained
in~\cite{D-D}*{Theorem~1}, under extra regularity for $b(\cdot)$.
Following their proof, it is possible to show that
\begin{multline*}
\|\ue-\ur\|_{H^1(\O)} \leq C\|\ue-w_r\|_{H^1(\O)}
\biggl(1+\inf_{\tilde\chi\in
V_r}\|\phi-\tilde\chi\|_{H^1(\O)}+\|\ue-w_r\|_{H^1(\O)}^2\biggr)
\\
+C\|\ue-\ur\|_{L^2(\O)} \biggl(\inf_{\tilde\chi\in
V_r}\|\phi-\tilde\chi\|_{H^1(\O)}+\|\ue-\ur\|_{L^2(\O)}^2\biggr),
\end{multline*}
for all $w_r\in V_r$, where $\phi$ is the solution of a linear dual problem. It follows then that $\|\ue-\ur\|_{L^2(\O)}$ is small enough as long as the mesh size $h$ is small enough, and a best approximation  result follows. However, the compactness argument of~\cite{D-D} does not allow, in principle, the mesh size
to be independent of the small scales.
\end{remark}

Finally, strict monotonicity is also a sufficient condition for the best approximation result of Lemma~\cite{evans}, i.e,
\begin{equation*}
\int_\O\aex[b(v_r)\grad v_r-b(w_r)\grad w_r]\cdot\grad(v_r-w_r)\,dx\ge c\|v_r-w_r\|_{H^1(\O)}^2
\end{equation*}
for all $v_r$, $w_r\in V_r$. In this case,
\begin{multline*}
\|\ur-w_r\|_{H^1(\O)}^2
\le c\int_\O\aex[b(u_r)\grad u_r-b(w_r)\grad w_r]\cdot\grad(u_r-w_r)\,dx
\\
=c\int_\O\aex[b(\ue)\grad \ue-b(w_r)\grad
w_r]\cdot\grad(u_r-w_r)\,dx
\\
\le c\|b(\ue)\grad \ue-b(w_r)\grad w_r\|_{L^2(\O)}
\|\ur-w_r\|_{H^1(\O)}
\\
\le c\bigl(\|b(\ue)\grad \ue-b(w_r)\grad
\ue\|_{L^2(\O)} +\|b(w_r)\grad \ue-b(w_r)\grad
w_r\|_{L^2(\O)}\bigr) \|\ur-w_r\|_{H^1(\O)},
\end{multline*}
and we conclude that $\|\ur-w_r\|_{H^1(\O)}\le
c\|\ue-w_r\|_{H^1(\O)}$ for all $w_r\in V_r$. An estimate
as~\eqref{e:cea} follows from the triangle inequality.

\section{Possible Linearizations}\label{s:lineariz}
As in the original problem~\eqref{e:pnl}, the RFB
approximation~\eqref{e:rfb2}, or equivalently~\eqref{e:rfb}, is
still given by a nonlinear problem. We investigate here some ideas to linearize the problem. In the
next subsection, we investigate fixed point schemes, and in the
following subsection, we discuss a proposal named \emph{reduced
RFB}.

\subsection{Fixed point formulation}\label{ss:rfbpf}
A first idea to linearize the original problem~\eqref{e:pnl} is the following. Let $u_\e^0\in H_0^1(\Omega)$, and for $n\in\N$,  given $u_\e^{n-1}\in H_0^1(\Omega)$, compute $u_\e^n\in H_0^1(\Omega)$ as the solution of
\begin{equation}\label{e:interat}
\int_\Omega\alpha_\e(x)b(u_\e^{n-1})\grad(u_\e^n)\cdot\grad v\,dx
=\int_\Omega fv\,dx\quad\text{for all }v\in H_0^1(\Omega).
\end{equation}
In the context of the RFB method, we use~\eqref{e:rfb} to propose the following iterative scheme. Let $u_\e^0\in V_r$, and $n\in\N$. Given $u_r^{n-1}\in V_r$, compute $u_r^n\in V_r$ solution of
\begin{equation}\label{e:rfbinterat}
\int_\Omega\alpha_\e(x)b(u_r^{n-1})\grad(u_r^n)\cdot\grad v_r\,dx
=\int_\Omega fv_r\,dx\quad\text{for all }v_r\in V_r.
\end{equation}
Observe that the above scheme discretizes~\eqref{e:interat}. Hence, \emph{discretization} and \emph{linearization} commutes. Since the problem now is linear, we head back to the situation described in Remark~\ref{obslinear}.

We can also rewrite~\eqref{e:rfbinterat} in terms of global/local problems. Given $u_h^{n-1}\in V_h$ and $u_b^{n-1}\in V_b$, find $u_h^n\in V_h$ and $u_b^n\in V_b$ such that
\begin{equation}\label{e:rfbinteratexp}
\begin{gathered}
\int_\Omega\alpha_\e(x)b(u_h^{n-1}+u_b^{n-1})\grad(u_h^n+u_b^n)\cdot\grad v_h\,dx
=\int_\Omega fv_h\,dx,
\\
-\div[\alpha_\e(x)b(u_h^{n-1}+u_b^{n-1})\grad(u_h^n+u_b^n)]=f\quad\text{in }K,
\end{gathered}
\end{equation}
for all $v_h\in V_h$ and all $K\in\T_h$.

\begin{lemma}Given $u_\e^0\in H_0^1(\O)$ and $u_r^0\in V_r$, let $u_\e^n\in H_0^1(\O)$ and $u_r^n\in V_r$ be defined from~\eqref{e:interat} and~\eqref{e:rfbinterat} for $n\in\N$. Then $\lim_{n\to\infty}u_\e^n=u_\e$ and $\lim_{n\to\infty}u_r^n=u_r$ in $H_0^1(\O)$.
\end{lemma}
\begin{proof}
We first consider the continuous problem, for a fixed $\e$. Note that $\|\grad u_\e^n\|_0\le c\|f\|_{-1}$, and then $\|\grad u_\e^n\|_{0,\O}$ is bounded. Therefore, there exist $\bar u\in H_0^1(\O)$ and a subsequence of $u_\e^n$, indexed by $n\in\N$, but still denoted by $u_\e^n$, such that $u_\e^n$ weakly converges to $\bar u$ in $H_0^1(\O)$, with strong convergence in $L^2(\O)$. Thus, from the Lebesgue Dominated Convergence Theorem, $b(u_\e^n)\grad v\to b(\bar u)\grad v$ strongly in $L^2(\O)$, for all $v\in H_0^1(\O)$. Note also that  $\int_\O\grad(u_\e^n-\bar u)\cdot\btau\,dx\to0$ for all $\btau\in L^2(\O)$. Indeed, from Helmholtz decomposition, there exist $p\in H_0^1(\O)$, $q\in H^1(\O)$ such that $\btau=\grad p+\curl q$. Therefore,
\[
\int_\O\grad(u_\e^n-\bar u)\cdot\btau\,dx
=\int_\O\grad(u_\e^n-\bar u)\cdot\grad p\,dx\to0
\]
as $n\to\infty$. It follows from these results that, for all $v\in H_0^1(\O)$,
\begin{multline*}
\int_\O[b(u_\e^{n-1})\grad u_\e^n-b(\bar u)\grad\bar u]\grad v
\\
=\int_\O[b(u_\e^{n-1})-b(\bar u)]\grad u_\e^n\grad v
+\int_\O b(\bar u)[\grad u_\e^n-\grad\bar u]\grad v
\\
\le\|[b(u_\e^{n-1})-b(\bar u)]\grad v\|_0\|\grad u_\e^n\|_0+\int_\O b(\bar u)[\grad u_\e^n-\grad\bar u]\grad v.
\end{multline*}
Taking $n\to\infty$ we gather that
\begin{equation}\label{e:convprod}
\int_\O[b(u_\e^{n-1})\grad u_\e^n-b(\bar u)\grad\bar u]\grad v\to0.
\end{equation}
Thus
\[
0=\lim_{n\to\infty}\int_\O b(u_\e^{n-1})\grad u_\e^n\grad v-fv\,dx
=\int_\O b(\bar u)\grad\bar u\grad v-fv\,dx.
\]
Then $\bar u$ solves~\eqref{e:pnl}. From uniqueness of solutions, $\bar u=u_\e$, and the whole sequence, and not only a subsequence, $u_\e^n$ converges to $\bar u$.

To show that the convergence is actually strong,
note~\cite{MR1477663} that
\begin{multline*}
\|\ue^n-\bar u\|_{H^1(\O)} \le\int_\O\alpha_\e
b(\ue^{n-1})\grad(\ue^n-\bar
u)\cdot\grad(\ue^n-\bar u)\,dx
\\
=\int_\O\alpha_\e b(\ue^{n-1})\grad(\bar u)\cdot\grad(\bar
u-2\ue^n)\,dx+\int_\O\alpha_\e
b(\ue^{n-1})\grad(\ue^n)\cdot\grad(\ue^n)\,dx
\\
=\int_\O\alpha_\e b(\ue^{n-1})\grad(\bar u)\cdot\grad(\bar
u-2\ue^n)\,dx+\int_\O f\ue^n\,dx
\\
\to-\int_\O\alpha_\e b(\bar u)\grad(\bar u)\cdot\grad(\bar u)\,dx+\int_\O f\bar u\,dx
\end{multline*}
since~\eqref{e:convprod} holds. Thus the convergence
$\ue^n\to\bar u$ is strong in $H^1(\O)$.

The second part of the lemma, regarding the RFB approximation, follows from basically the same arguments. Since $\|\grad u_r^n\|_0\le c\|f\|_{-1}$, there exists $\bar u_r\in H_0^1(\O)$ and a subsequence still denoted by $u_r^n$ such that $u_r^n$ weakly converges to $\bar u_r$ in $H_0^1(\O)$, whereas strong convergence holds in $L^2(\O)$. Again, $b(u_r^n)\grad v\to b(\bar u_r)\grad v$ strongly in $L^2(\O)$, for all $v\in H_0^1(\O)$. Note also that $\int_\O\grad(u_r^n-\bar u_r)\cdot\btau\,dx\to0$ for all $\btau\in L^2(\O)$. Indeed, from Helmholtz decomposition, there exist $p\in H_0^1(\O)$, $q\in H^1(\O)$ such that $\btau=\grad p+\curl q$. Thus
\[
\int_\O\grad(u_r^n-\bar u_r)\cdot\btau\,dx
=\int_\O\grad(u_r^n-\bar u_r)\cdot\grad p\,dx\to0
\]
as $n\to\infty$. From these results, we gather that for all $v\in H_0^1(\O)$,
\begin{multline*}
\int_\O[b(u_r^{n-1})\grad u_r^n-b(\bar u_r)\grad\bar u_r]\grad v
\\
=\int_\O[b(u_r^{n-1})-b(\bar u_r)]\grad u_r^n\grad v
+\int_\O b(\bar u_r)[\grad u_r^n-\grad\bar u_r]\grad v
\\
\le\|[b(u_r^{n-1})-b(\bar u_r)]\grad v\|_0\|\grad u_r^n\|_0+\int_\O b(\bar u_r)[\grad u_r^n-\grad\bar u_r]\grad v.
\end{multline*}
Taking $n\to\infty$, it follows that $\int_\O[b(u_r^n)\grad u_r^n-b(\bar u_r)\grad\bar u_r]\grad v\to0$ for all $v\in H_0^1(\O)$. Considering now $v\in V_r$, we have that
\[
0=\lim_{n\to\infty}\int_\O b(u_r^{n-1})\grad u_r^n\grad v-fv\,dx
=\int_\O b(\bar u_r)\grad\bar u_r\grad v-fv\,dx.
\]
Since $V_r$ is closed, $\bar u_r\in V_r$. Therefore $\bar u_r=u_r$ solves~\eqref{e:rfb}. If uniqueness also holds, the whole sequence $u_\e^n$ converges to $\bar u$.
\end{proof}

\begin{lemma}Given $u_\e^0\in H_0^1(\O)$ and $u_r^0\in V_r$, let $u_\e^n\in H_0^1(\O)$ and $u_r^n\in V_r$ be defined by~\eqref{e:interat} and~\eqref{e:rfbinterat}, $n\in\N$. Then, if $\ue$ is sufficiently small in $W^{1,6}(\Omega)$, we have that
\[
\|\ue^n-\ue\|_{H^1(\O)}+\|\ur^n-\ur\|_{H^1(\O)}\le \bar\alpha\|\ue^{n-1}-\ue\|_{H^1(\O)},
\]
for $\bar\alpha<1$.
\end{lemma}
\begin{proof}
Note that
\begin{multline*}
\|\ue^n-\ue\|_{H^1(\O)}^2
\le\int_\O\aex b(\ue^{n-1})\grad(\ue^n-\ue)\grad(\ue^n-\ue)\,dx
\\
=\int_\O\aex[-b(\ue^{n-1})+b(\ue)]\grad\ue\grad(\ue^n-\ue)\,dx
\\
\le c\|\ue^{n-1}-\ue\|_{L^2(\O)}^{1/2}\|\ue^{n-1}-\ue\|_{H^1(\O)}^{1/2}
\|\grad\ue\|_{L^6(\O)}\|\ue^n-\ue\|_{H^1(\O)}.
\end{multline*}
The result for $\ur^n$ is analogous.
\end{proof}

We end this subsection with an alternative linearization proposal, based on~\eqref{e:rfbinteratexp}. Given $u_h^{n-1}\in V_h$ and $u_b^{n-1}\in V_b$, find $u_h^n\in V_h$ and $u_b^n\in V_b$ such that
\begin{gather}
\int_\Omega\alpha_\e(x)b(u_h^{n-1}+u_b^{n-1})\grad(u_h^n+u_b^{n-1})
\cdot\grad v_h\,dx=\int_\Omega fv_h\,dx, \label{e:rfbinteratexpi}
\\
-\div[\alpha_\e(x)b(u_h^n+u_b^{n-1})\grad(u_h^n+u_b^n)]=f
\label{e:rfbinteratexpii}
\quad\text{in }K,
\end{gather}
for all $v_h\in V_h$ and all $K\in\T_h$. Observe that the above system is not coupled as in~\eqref{e:rfbinteratexp}. It is possible to solve first~\eqref{e:rfbinteratexpi} and only then solve~\eqref{e:rfbinteratexpii}.

\subsection{Reduced Residual Free Bubble Formulation}\label{ss:rfbred}
The idea here is to use the approximation $b(u_h+u_b)\approx b(u_h)$ at the local problem of the second equation in~\eqref{e:rfb2}. This induces a linearization that makes static condensation possible. In this case, we search for the approximation $\tilde\ur=\tilde u_h+\tilde u_b\in V_r$ such that
\begin{gather}
\int_\Omega\alpha_\e(x)b(\tilde u_h+\tilde u_b)\grad(\tilde u_h+\tilde u_b)\cdot\grad v_h\,dx
=\int_\Omega fv_h\,dx,\nonumber
\\
-\div[\alpha_\e(x)b(\tilde u_h)\grad\tilde u_b]
=f+\div[\alpha_\e(x)b(\tilde u_h)\grad\tilde u_h]
\quad\text{in }K, \label{e:rfblr}
\end{gather}
for all $v_h\in V_h$ and all $K\in\T_h$. Thus, the local problem~\eqref{e:rfblr} is linear with respect to $\tilde u_b$.

\begin{remark} Since~\eqref{e:rfblr} is linear, we can split $\tilde u_b=\tilde u_b^l+\tilde u_b^f$ in two parts, each solving~\eqref{e:rfblr} with $f$ and  $\div[\alpha_\e(x)b(\tilde u_h)\grad\tilde u_h]$ on the right hand side. However, the local and global problems are still coupled. The local problems for the MsFEM involve $\tilde u_b^l$ only, and to make the method cheaper, it is possible to replace $b(u_h)$ by $b(\int_Ku_h(x)\,dx)$, as in~\cite{E-H-G}, or by $(b(u_h(x_K)))$ as in~\cite{CH-Y}, where $x_K$ is an interior point of the element. In this way,~\eqref{e:rfblr} reduces to a much simpler equation, given by
\begin{equation*}
-\div[\aex\grad\tilde\ub^l]= \div[\aex\grad\tilde\uh] \quad \text{in } K.
\end{equation*}
From the equation linearity, the computation of the local bubble $u_b^l$ is determined solving the corresponding problems associated to the basis functions.

However, such simplification is not possible for the RFB method, due to the presence of the $\tilde u_b^f$ term. Such extra term is important since it can significantly improve the quality of the approximation~\cites{MR2203943,MR2142535,SG} in some situations.
\end{remark}

\end{document}